\documentclass{amsart}

\usepackage{amsfonts, amsmath, amssymb, amsthm}
\usepackage{mathrsfs}

\newcommand{\st}{\,\big|\,}
\newcommand{\real}{\mathbb{R}}

\newcommand{\tor}{\mathbb{T}}
\newcommand{\integer}{\mathbb{Z}}
\newcommand{\rational}{\mathbb{Q}}

\newtheorem{theorem}{Theorem}
\newtheorem{lemma}[theorem]{Lemma}

\newtheorem{question}[theorem]{Question}

\begin{document}

\title[ABELIAN SUBGROUPS IN NO CONJUGATE POINTS]{ABELIAN SUBGROUPS OF THE FUNDAMENTAL GROUP OF A SPACE WITH NO CONJUGATE POINTS}

\author[JAMES DIBBLE]{JAMES DIBBLE}
\address{Department of Mathematics, University of Iowa, Iowa City, IA 52242}
\email{james-dibble@uiowa.edu}
\thanks{I'm grateful to Vitali Kapovitch, Michael Kapovich, and Christopher Croke for helpful discussions. This topic arose during a conversation with Vitali Kapovitch, who showed using Corollary 4.3 of \cite{IvanovKapovitch2014} that, for a sufficiently regular Riemannian manifold, the center of its fundamental group must be finitely generated. The proof of Theorem 1 in the second section contains a simplification of my original argument due to an anonymous referee, whose improvement works without any regularity assumptions.}


\date{}

\begin{abstract}
Each Abelian subgroup of the fundamental group of a compact and locally simply connected $d$-dimensional length space with no conjugate points is isomorphic to $\integer^k$ for some $0 \leq k \leq d$. It follows from this and previously known results that each solvable subgroup of the fundamental group is a Bieberbach group. In the Riemannian setting, this may be proved using a novel property of the asymptotic norm of each Abelian subgroup.
\end{abstract}

\maketitle

\section{Introduction}

A locally simply connected length space $X$ with universal cover $\pi : \hat{X} \to X$ has \textbf{no conjugate points} if any two points in $\hat{X}$ can be joined by a unique geodesic. Let $X$ be a compact and locally simply connected length space with no conjugate points and finite Hausdorff dimension $d$. In the Riemannian case, it has been believed for some time that Abelian subgroups of $\pi_1(X)$ must be finitely generated; for example, this is stated in \cite{CrokeSchroeder1986}, although the argument there contains a gap. It will be shown here that each Abelian subgroup is isomorphic to $\integer^k$ for some $0 \leq k \leq d$.

\begin{theorem}\label{main theorem}
    Each Abelian subgroup of $\pi_1(X)$ is isomorphic to $\integer^k$ for some $0 \leq k \leq d$.
\end{theorem}

\noindent For nonpositively curved manifolds, this is a consequence of the flat torus theorem of Gromoll--Wolf \cite{GromollWolf1971} and Lawson--Yau \cite{LawsonYau1972}, which was generalized to manifolds with no focal points by O'Sullivan \cite{O'Sullivan1976}.

It was proved by Yau \cite{Yau1971} in the case of nonpositive curvature, and O'Sullivan \cite{O'Sullivan1976} for no focal points, that every solvable subgroup of the fundamental group is a Bieberbach group. Croke--Schroeder \cite{CrokeSchroeder1986} mapped out a way to generalize this to spaces with no conjugate points: If a torsion-free solvable group has the property that its Abelian subgroups are all finitely generated and straight, then it must be a Bieberbach group. Lebedeva \cite{Lebedeva2002} showed that finitely generated Abelian subgroups of the fundamental group of a compact and locally simply connected length space with no conjugate points must be straight. Combining this with Theorem \ref{main theorem} completes the argument set out by Croke--Schroeder.

\begin{theorem}\label{solvable subgroups}
    Each solvable subgroup of $\pi_1(X)$ is a Bieberbach group.
\end{theorem}

\noindent This continues the theme, developed in \cite{CrokeSchroeder1986}, \cite{Lebedeva2002}, \cite{IvanovKapovitch2014}, and unpublished work of Kleiner that, at the level of fundamental group, spaces with no conjugate points resemble those with nonpositive curvature.

Since the exponential map at each point of its universal cover is a diffeomorphism, a Riemannian manifold with no conjugate points must be aspherical. It's worth pointing out that this condition isn't enough to guarantee the conclusion of Theorem \ref{main theorem}, as Mess \cite{Mess1990,Luck2012} showed that for each $n \geq 4$ there exists a compact manifold with universal cover $\real^n$ whose fundamental group contains a divisible Abelian subgroup, which cannot be finitely generated.

The second section contains a short proof of Theorem \ref{main theorem}. The third section gives a different proof in the Riemannian setting, based on a property of Riemannian norms satisfied by the asymptotic norm of each Abelian subgroup of the fundamental group.

\section{Proof of Theorem \ref{main theorem}}

Fix $\hat{p} \in \hat{X}$ and a basepoint $p = \pi(\hat{p})$ for $\pi_1(X)$. Overloading notation, each $\gamma \in \pi_1(X)$ will be identified with the corresponding deck transformation of $\hat{X}$. Let $\Gamma$ be an Abelian subgroup of $\pi_1(X)$, in which the group operation is written additively, and suppose $\sigma_1,\ldots,\sigma_k \in \Gamma$ are linearly independent. Denote by $G$ the subgroup generated by the $\sigma_i$. The following are proved in \cite{Lebedeva2002}: On $\pi_1(X)$, the function
\[
    | \gamma |_\infty = \lim_{m \to \infty} \frac{\hat{d}(m\gamma(\hat{p}),\hat{p})}{m}
\]
is positively homogeneous over $\integer$. It is bounded below on $\pi_1(X) \setminus \{ e \}$ by $\mathrm{sys}(X)$, the length of the shortest nontrivial geodesic loop in $X$, so $\pi_1(X)$ is torsion free. Its restriction to $\Gamma$ satisfies the triangle inequality, and, with respect to the isomorphism $G \cong \integer^k$ that takes each $\sigma_i$ to the $i$-th standard basis vector, $| \cdot |_\infty$ extends to a norm $\| \cdot \|_\infty$ on $\real^k$.

Denote by $\| \cdot \|$ the Euclidean norm on $\real^k$. From the identifications $G(\hat{p}) \cong G \cong \integer^k$, $G(\hat{p})$ inherits the coordinate functions $\rho_1,\ldots,\rho_k$ on $\integer^k$. Since $\| \cdot \|_\infty$ is a norm on $\real^k$, there exists $C > 0$ such that
\[
    \frac{1}{C} \| u \|_\infty \leq \| u \| \leq C \| u \|_\infty 
\]
for all $u \in \real^k$. The number $C$ is a Lipschitz constant for the $\rho_i$ on $G(\hat{p})$, and, as in the proof of Kirszbraun's theorem \cite{Kirszbraun1934}, the functions
\[
    f_i(\hat{x}) = \min_{\gamma \in G}[\rho_i(\gamma(\hat{p})) + C \hat{d}(\hat{x},\gamma(\hat{p}))]
\]
are Lipschitz extensions of the $\rho_i$ to $\hat{N}$. Each $f_i$ is $(G,\integer)$-equivariant, in the sense that $f_i(\gamma(\hat{x})) - f_i(\hat{x}) \in \integer$ for all $\hat{x} \in \hat{N}$ and all $\gamma \in G$.

The map $f = (f_1,\ldots,f_k) : \hat{N} \to \real^k$ is Lipschitz, and $f(\gamma(\hat{x})) - f(\hat{x}) \in \integer^k$ for all $\hat{x} \in \hat{N}$ and all $\gamma \in G$. By construction, $f(G(\hat{p})) = \integer^k$. Since $G$ is Abelian, there exists a map $\phi : \tor^k \to X$ such that $\phi_*(\pi_1(\tor^k)) \cong G$. Lift $\phi$ to a map $\hat{\phi} : \real^k \to \hat{N}$. The composition $f \circ \hat{\phi} : \real^k \to \real^k$ descends to a map $\tor^k \to \tor^k$ with surjective induced homomorphism, so by degree theory it must be surjective. Thus $f$ is surjective. Since a Lipschitz map cannot increase Hausdorff dimension, $k \leq d$.

It follows that $\Gamma$ has rank at most $d$. If it has rank zero, then the result is trivial. Without loss of generality, suppose it has rank $k > 0$. For any $\gamma \in \Gamma$, there exist $n,a_1,\ldots,a_k \in \integer$ such that $n\gamma = \sum_{i=1}^k a_i \sigma_i$. It is well known that the function $F : \Gamma \to \rational^k$ defined by $F(e) = (0,\ldots,0)$ and
\[
    F(\gamma) = (a_1/n,\ldots,a_k/n)
\]
for $\gamma \neq e$ is a well-defined and injective homomorphism, so $F$ is an isomorphism onto its image $\Gamma_0$. This map satisfies
\begin{align*}
    \| F(\gamma) \|_\infty &= \| (a_1/n,\ldots,a_k/n) \|_\infty = \frac{1}{|n|} \big\| (a_1,\ldots,a_k) \big\|_\infty\\
        &= \frac{1}{|n|} \big| \sum_{i=1}^k a_i \sigma_i \big|_\infty = \frac{1}{|n|} | n\gamma |_\infty = | \gamma |_\infty
\end{align*}
for any $\gamma \neq e$. For any distinct $q_0,q_1 \in \Gamma_0$, there exist distinct $\gamma_0,\gamma_1 \in \Gamma$ such that $F(\gamma_i) = q_i$ for each $i$. For $c = 1/C$, one has that
\begin{align*}
    \| q_0 - q_1 \| &\geq c\| q_0 - q_1 \|_\infty = c\|F(\gamma_0) - F(\gamma_1)\|_\infty = c\|F(\gamma_0 - \gamma_1)\|_\infty\\
        &= c|\gamma_0 - \gamma_1|_\infty \geq c \cdot \mathrm{sys}(X) > 0\textrm{.}
\end{align*}
Thus $\Gamma_0$ is a discrete subgroup of $\real^k$, and, consequently, $\Gamma \cong \integer^k$.

\section{Busemann functions in the Riemannian setting}

For simplicity, it will be assumed in this section that $X$ is a smooth $d$-dimensional Riemannian manifold, although what follows holds when $X$ is $C^r$ for some $r$ depending on $d$. As before, let $G$ be an Abelian subgroup of $\pi_1(X)$ generated by linearly independent $\gamma_1,\ldots,\gamma_k$. The key step in the proof of Theorem \ref{main theorem} is the construction of a $(G,\integer^k)$-equivariant map $f : \hat{X} \to \real^k$ such that $f(G(\hat{p})) = \integer^k$. When $X$ is Riemannian, another such map may be constructed using a nondegenerate collection of Busemann functions.

An important theorem of Ivanov--Kapovitch \cite{IvanovKapovitch2014} states that, whenever $\alpha_1,\alpha_2 \in \pi_1(X)$ commute, the change in the Busemann functions of axes of $\alpha_2$ under the action of $\alpha_1$ is constant on $\hat{X}$. This was previously proved by Croke--Schroeder \cite{CrokeSchroeder1986} for analytic $X$. Thus one may define a function $B : G \times G \to \real$ by setting $B(\alpha_1,\alpha_2)$ equal to that change.

Because $B(\alpha,\alpha) = | \alpha |_\infty^2$ for all $\alpha \in G$, one might hope to show that $B$ extends to an inner product and, consequently, that $\| \cdot \|_\infty$ is Riemannian. In fact, $B$ satisfies a number of the properties of an inner product: It is linear over $\integer$ in the first slot (see Corollary 4.2 of \cite{IvanovKapovitch2014}), $B(\alpha_1,n\alpha_2) = nB(\alpha_1,\alpha_2)$ for all $n \in \integer$, and it satisfies a version of the Cauchy--Schwarz inequality,
\begin{equation}\label{cauchy--schwarz}
    |B(\alpha_1,\alpha_2)| \leq | \alpha_1 |_\infty | \alpha_2 |_\infty\textrm{,}
\end{equation}
with equality if and only if $\alpha_1$ and $\alpha_2$ are rationally related. It follows that $B$ extends to an inner product if and only if it is symmetric, but it's far from clear that symmetry holds in general (cf. \cite{BuragoIvanov1994}). Regardless, $B$ also resembles an inner product in the following way.

\begin{lemma}\label{main lemma}
    For each $1 \leq m \leq k$, there exist $\alpha_1,\ldots,\alpha_m \in \mathrm{span}\,\{\gamma_1,\ldots,\gamma_m\}$ such that the $m \times m$ matrix $[B(\alpha_i,\alpha_j)]$ is nonsingular.
\end{lemma}

\noindent If $\alpha_1,\ldots,\alpha_k$ are as in Lemma \ref{main lemma} and $b_1,\ldots,b_k$ are Busemann functions of respective axes, then up to composition with an affine isomorphism the map $F = (b_1,\ldots,b_k) : \hat{X} \to \real^k$ is $(G,\integer^k)$-equivariant and satisfies $F(G(\hat{p})) = \integer^k$. The Riemannian version of Theorem \ref{main theorem} follows.

The proof of Lemma \ref{main lemma} is by induction. When $m = 1$, the conclusion holds with $\alpha_1 = \gamma_1$. Suppose the conclusion holds for some $1 \leq m < k$. If the conclusion fails when $\alpha_{m+1} = \gamma_{m+1}$, then there exists a nonzero $c = (c_1,\ldots,c_{m+1})$ in the null space of the $(m+1) \times (m+1)$ matrix $[B(\alpha_j,\alpha_i)]$. The following lemma then completes the inductive step.

\begin{lemma}\label{solid cone}
    There exists a solid cone $C$ centered around the ray $\{ rc \st r \geq 0 \}$ such that, if $x = (x_1,\ldots,x_{m+1}) \in C \cap \integer^{m+1}$, $\tilde{\alpha}_i = \alpha_i$ for $1 \leq i \leq m$, and $\tilde{\alpha}_{m+1} = \sum_{i=1}^{m+1} x_i \alpha_i$, then the $(m+1) \times (m+1)$ matrix $[B(\tilde{\alpha}_i,\tilde{\alpha}_j)]$ is nonsingular.
\end{lemma}

\noindent The proof of Lemma \ref{solid cone} uses the following elementary fact.

\begin{lemma}\label{linear algebra lemma}
    Let $A,C > 0$. Suppose $M_\ell$ is a sequence of $(p+1) \times q$ matrices of the form
    \[
        \left[ \begin{array}{c} M \\ b_\ell \end{array} \right]
    \]
    for a fixed $p \times q$ matrix $M$ and a sequence $b_\ell \in \real^q$ such that $\| b_\ell \| \to 0$. Suppose also that $w_\ell$ is a sequence of vectors in $\real^{p+1}$ of the form
    \[
        \left[ \begin{array}{c} a_\ell \\ C_\ell \end{array} \right]
    \]
    for $a_\ell \in \real^p$ satisfying $\| a_\ell \| \leq A$ and $|C_\ell| \geq C$. If $v_\ell \in \real^q$ satisfy $M_\ell v_\ell = w_\ell$, then $\|M(v_\ell/\|v_\ell\|)\| \to 0$. Consequently, $M$ has nontrivial null space.
\end{lemma}

\begin{proof}[Proof of Lemma \ref{solid cone}]
    Without loss of generality, one may suppose that $\max |c_i| = 1$. Assume for the sake of contradiction that the result is false. Then, for each $i$ and any fixed sequence $\varepsilon_\ell \searrow 0$, there exists a sequence of rational numbers $p_i^\ell/q_i^\ell$ such that $|c_i - p_i^\ell/q_i^\ell| < \varepsilon_\ell$ and, when $\tilde{\alpha}_i^\ell = \alpha_i$ for $1 \leq i \leq m$ and $\tilde{\alpha}_{m+1}^\ell = \sum_{i=1}^{m+1} (\prod_{j \neq i} q_j^\ell) p_i^\ell \alpha_i$, each $(m+1) \times (m+1)$ matrix $M _\ell = [B(\tilde{\alpha}_i^\ell,\tilde{\alpha}_j^\ell)]$ is singular.

    Let $W = [B(\alpha_j,\alpha_i)]$ for $1 \leq i,j \leq m + 1$, and write
    \begin{equation}\label{equation1}
    \begin{aligned}
        w_\ell &= W \big( (\prod_{j \neq 1} q_j^\ell)p_1^\ell,\ldots,(\prod_{j \neq m+1} q_j^\ell) p_{m+1}^\ell \big)\\
        &= \big( \sum_{i=1}^{m+1} (\prod_{j \neq i} q_j^\ell) p_i^\ell B(\alpha_i,\alpha_1),\ldots,\sum_{i=1}^{m+1} (\prod_{j \neq i} q_j^\ell) p_i^\ell B(\alpha_i,\alpha_{m+1}) \big)\\
        &= (B(\tilde{\alpha}_{m+1}^\ell,\tilde{\alpha}_1^\ell),\ldots,B(\tilde{\alpha}_{m+1}^\ell,\tilde{\alpha}_m^\ell),B(\tilde{\alpha}_{m+1}^\ell,\alpha_{m+1}))\textrm{.}
    \end{aligned}
    \end{equation}
    Let $K = \displaystyle \max_{1 \leq i,j \leq m+1} |B(\alpha_i,\alpha_j)|$. Then
    \begin{equation}\label{equation2}
    \begin{aligned}
        \| w_\ell \| &= \| (\prod_j q_j^\ell) W(p_1^\ell/q_1^\ell,\ldots,p_{m+1}^\ell/q_{m+1}^\ell) \|\\
            &\leq | \prod_j q_j^\ell | K\varepsilon_\ell \sqrt{m+1}\textrm{.}
    \end{aligned}
    \end{equation}
    The inductive hypothesis and the linearity of $B$ in the first slot imply that $\alpha_1,\ldots,\alpha_{m+1}$ are linearly independent. The word norm of $\tilde{\alpha}_{m+1}^\ell$ with respect to the subgroup of $H$ generated by $\alpha_1,\ldots,\alpha_{m+1}$ is
    \[
        | \tilde{\alpha}_{m+1}^\ell |_{\mathrm{word}} = \sum_{i=1}^{m+1} |\prod_{j \neq i} q_j^\ell||p_i^\ell|\textrm{.}
    \]
    Because the corresponding norms on $\real^{m+1}$ are equivalent, there exists $D > 0$, depending only on $\alpha_1,\ldots,\alpha_{m+1}$, such that
    \[
        \frac{1}{D} \sum_{i=1}^{m+1} |\prod_{j \neq i} q_j^\ell||p_i^\ell| \leq | \tilde{\alpha}_{m+1}^\ell |_\infty \leq D \sum_{i=1}^{m+1} |\prod_{j \neq i} q_j^\ell||p_i^\ell|\textrm{.}
    \]
    By the Cauchy--Schwarz inequality \eqref{cauchy--schwarz}, for each $1 \leq i \leq m$,
    \begin{equation}\label{equation3}
        |B(\tilde{\alpha}_i^\ell,\tilde{\alpha}_{m+1}^\ell)| \leq | \tilde{\alpha}_i^\ell |_\infty | \tilde{\alpha}_{m+1}^\ell |_\infty \leq D\sqrt{K} \sum_{i=1}^{m+1} |\prod_{j \neq i} q_j^\ell||p_i^\ell|\textrm{.}
    \end{equation}
    Similarly,
    \begin{equation}\label{equation4}
        B(\tilde{\alpha}_{m+1}^\ell,\tilde{\alpha}_{m+1}^\ell) = | \tilde{\alpha}_{m+1}^\ell |_\infty^2 \geq (1/D^2)[\sum_{i=1}^{m+1} |\prod_{j \neq i} q_j^\ell||p_i^\ell|]^2\textrm{.}
    \end{equation}
    Let $a_\ell = (B(\tilde{\alpha}_1^\ell,\tilde{\alpha}_{m+1}^\ell),\ldots,B(\tilde{\alpha}_m^\ell,\tilde{\alpha}_{m+1}^\ell))$, $b_\ell = (B(\tilde{\alpha}_{m+1}^\ell,\tilde{\alpha}_1^\ell),\ldots,B(\tilde{\alpha}_{m+1}^\ell,\tilde{\alpha}_m^\ell))$, $c_\ell = B(\tilde{\alpha}_{m+1}^\ell,\tilde{\alpha}_{m+1}^\ell)$, and $M = [B(\alpha_i,\alpha_j)]$ for $1 \leq i,j \leq m$. Write
    \[
        \tilde{a}_\ell = a_\ell/[\sum_{i=1}^{m+1}|\prod_{j \neq i} q_j^\ell||p_i^\ell|]\textrm{,}
    \]
    \[
        \tilde{b}_\ell = b_\ell/|\prod_j q_j^\ell|\textrm{,}
    \]
    and
    \[
        \tilde{c}_\ell = c_\ell/[|\prod_j q_j^\ell| \sum_{i=1}^{m+1}|\prod_{j \neq i} q_j^\ell||p_i^\ell|]\textrm{.}
    \]
     By \eqref{equation1} and \eqref{equation2}, $\| \tilde{b}_\ell \| \leq \| w_\ell \|/|\prod_j q_j^\ell| \leq K\varepsilon_\ell \sqrt{m+1}$; by \eqref{equation3}, $\| \tilde{a}_\ell \| \leq D\sqrt{mK}$; and, by \eqref{equation4}, $\tilde{c}_\ell \geq 1/(2D^2)$ for all sufficiently large $\ell$. Since $M$ is nonsingular, it follows from Lemma \ref{linear algebra lemma} that the matrices
    \[
        \left[ \begin{array}{cc} M & \tilde{a}_\ell \\ \tilde{b}_\ell & \tilde{c}_\ell \end{array} \right]
    \]
    are nonsingular for all such $\ell$. The corresponding $M_\ell$ must also be nonsingular, which is a contradiction.
\end{proof}

When $m = 2$ in Lemma \ref{main lemma}, inequality \eqref{cauchy--schwarz} implies that one may take $\alpha_1 = \gamma_1$ and $\alpha_2 = \gamma_2$. When $X$ has no focal points, one may, by the flat torus theorem, take $\alpha_i = \gamma_i$ for all $i$. However, in the general case for $m \geq 3$, there is no apparent local structure that forces the Busemann functions of the axes of the $\gamma_i$ to have linearly independent gradients, and it is not clear that the conclusion of Lemma \ref{main lemma} holds with $\alpha_i = \gamma_i$ for all $i$.

\begin{question}
    Must the $k \times k$ matrix $[B(\gamma_i,\gamma_j)]$ be nonsingular?
\end{question}

\bibliographystyle{amsplain}
\bibliography{bibliography}

\end{document}